\documentclass[11pt,a4paper,fullpage]{article}
\usepackage{amsmath,amsthm}
\usepackage{amssymb}

\newtheorem{theorem}{Theorem}

\newtheorem{lemma}[theorem]{Lemma}

\title{{\bf Explicit formulae for the mean value of products of values of Dirichlet $L$-functions\\ 
at positive integers}}
\author{
St\'ephane R. LOUBOUTIN\\
Aix Marseille Universit\'e, CNRS, Centrale Marseille, I2M,\\ 
Marseille, France\\
stephane.louboutin@univ-amu.fr}
\date{\today}
\begin{document}
\bibliographystyle{alpha}
\maketitle

\centerline{To appear in Proc. Amer. Math. Soc.}

\begin{abstract}
Let $m\ge 1$ be a rational integer. 
We give an explicit formula for the mean value 
$$\frac{2}{\phi(f)}\sum_{\chi (-1)=(-1)^m}\vert L(m,\chi )\vert^2,$$ 
where $\chi$ ranges over the $\phi (f)/2$ Dirichlet characters modulo $f>2$ with the same parity as $m$. 
We then adapt our proof to obtain explicit means values for products of the form
$L(m_1,\chi_1)\cdots L(m_{n-1},\chi_{n-1})\overline{L(m_n,\chi_1\cdots\chi_{n-1})}$.
\end{abstract}

\section{Introduction}
For $m\geq 1$ and $f>2$, let 
$$M(m,f)
=\frac{2}{\phi (f)}
\sum_{\chi (-1)=(-1)^m}\vert L(m,\chi )\vert^2$$
denote the mean value of $\vert L(m,\chi )\vert^2$, 
where $\chi$ ranges over the $\phi (f)/2$ Dirichlet characters modulo $f$ such that $\chi (-1) =(-1)^m$, 
i.e. over the Dirichlet characters modulo $f$ of the same parity as $m$. 

The first step towards the study of these $M(m,f)$'s was made in \cite{Wal}, 
where the formula for $M(1,f)$ was obtained, but only for prime moduli $f$. 
Secondly, in \cite{Qi} and \cite{Lou93} the formula for $M(1,f)$ for non necessarily prime moduli $f$ 
was obtained. 
Thirdly, in \cite{LouCollMath90} a method was developed for obtaining recursively 
a formula for $M(m,f)$ for a given $m$. 
This method was applied to obtain the formulae below for $M(m,f)$ for $m=1$, $2$, $3$ and $4$. 
Fourthly, in \cite{Lin} the author developed another way for obtaining recursively 
a formula for $M(m,f)$ for a given $m$. 
S. Louboutin and X. Lin's methods are easy to use with any software for mathematical computation 
for obtaining a formula for $M(m,f)$ for a given $m$. 
But they did not give an explicit formula for these $M(m,f)$'s. 
Fifthly, in \cite{LZ} the authors gave a rather complicated but explicit formula for $M(m,f)$ 
and the mean value of the product of the values of two Dirichlet $L$-functions at positive integers 
of the same parity.

Here, inspired by \cite[Proof of Theorem 2.1]{BY}, 
we obtain in Theorem \ref{mainresult} an explicit formula for these $M(m,f)$'s. 
Then, we extend our method developed to obtain Theorem \ref{mainresult} to get in Theorem \ref{mainresultgeneralization} an explicit formula for the mean value of more general products of the form 
$L(m_1,\chi_1)\cdots L(m_{n-1},\chi_{n-1})\overline{L(m_n,\chi_1\cdots\chi_{n-1})}$. 
In particular, for $n=2$ we obtain in Theorem \ref{M(m,n;f)} a much simpler formula 
than the one given in \cite[Theorem 1.1]{LZ}
for the mean value of the product of the values of two Dirichlet $L$-functions at positive integers 
of the same parity.

\begin{theorem}\label{mainresult}
For $m=1$ we have
\begin{equation}\label{M(1,f)}
M(1,f)
=\frac{\pi^2}{6}\left (\phi_2(f)-3\frac{\phi_1(f)}{f}\right )
\end{equation} 
and for $m\geq 2$ we have
$$M(m,f)
=\zeta(2m)
\times\Biggl (\phi_{2m}(f)
+(-1)^{m-1}{2m\choose m}\sum_{k=1}^{\lfloor m/2\rfloor}
\frac{m}{m-k}{m\choose 2k}\frac{B_{2k}B_{2(m-k)}}{B_{2m}}
\frac{\phi_{2k}(f)}{f^{2m-2k}}
\Biggr ),$$
where $\phi_l(f)=\prod_{p\mid f}(1-1/p^l)$ 
and the Bernoulli rational numbers $B_k$ are given by 
$$\frac{t}{e^{t}-1}
=\sum_{k\geq 0}\frac{B_{k}}{k!}t^{k}.$$
In particular, $M(m,f)$ is asymptotic to $\zeta(2m)\phi_{2m}(f)$ 
as $f$ goes to infinity.
\end{theorem}

For example, in accordance with \cite[Theorem 2]{Lou93} and \cite[Page 371]{Lin} we have:
\begin{eqnarray*}
M(2,f)
&=&\frac{\pi^4}{90}\left (\phi_4(f)+10\frac{\phi_2(f)}{f^2}\right ),\\
M(3,f)
&=&\frac{\pi^6}{945}\left (\phi_6(f)-21\frac{\phi_2(f)}{f^4}\right ),\\
M(4,f)
&=&\frac{\pi^8}{9450}\left (\phi_8(f)+\frac{14}{3}\frac{\phi_4(f)}{f^4}+\frac{200}{3}\frac{\phi_2(f)}{f^6}\right ),\\
M(5,f)
&=&\frac{\pi^{10}}{93555}\left (\phi_{10}(f)-22\frac{\phi_4(f)}{f^6}-231\frac{\phi_2(f)}{f^8}\right ).
\end{eqnarray*}

\section{Proof of Theorem \ref{mainresult}}
Theorem \ref{mainresult} readily follows from the well known formula 
\begin{equation}\label{zeta(2m)}
\zeta(2m)
=(-1)^{m-1}\frac{(2\pi)^{2m}B_{2m}}{2\cdot (2m)!}.
\end{equation}
and from following Lemmas \ref{M(m,f)}, \ref{Rk} and \ref{TmSm}.

\begin{lemma}\label{M(m,f)}
(See \cite[Proposition 3 and Proof of Proposition 5]{LouCollMath90}).
Let $m\ge 1$ and $f>2$ be rational integers. 
Let $\cot^{(k)}$ denote the $k$th derivative of
$\cot:=\cos/\sin$. 
We have 
$$M(m,f)
=\frac{\pi^{2m}}{2\cdot ((m-1)!)^2f^{2m}}
\sum_{\substack{d\mid f\\ d>1}}
\mu (f/d)S_m(d),$$
where 
$$S_m(d)
:=\sum_{k=1}^{d-1}
\left (\cot^{(m-1)}(\pi k/d)\right )^2
\ \ \ \ \ (d>1).$$
\end{lemma}

\begin{lemma}\label{Rk}
If $R(X)=\sum_{l=0}^{k}r_{k,l}X^l\in {\mathbb Q}[X]$ is such that $R(1)=0$, 
 then for $f>2$ we have 
$$\sum_{\substack{d\mid f\\ d>1}}\mu (f/d)R(d)
=\sum_{d\mid f}\mu (f/d)R(d)
=\sum_{l=1}^{k}r_{k,l}\phi_l(f)f^l.$$
\end{lemma}

\begin{proof}
Notice that 
$\sum_{d\mid f}\mu (f/d)d^l 
=\sum_{d\mid f}\mu (d)(f/d)^l 
=f^l\phi_l(f)$.
\end{proof}

\begin{lemma}\label{TmSm}
Set $R_1(X) =(X^2-3X+2)/3$ and for $m\geq 2$ set
$$R_m(X)
=2^{2m}\cdot\left ((-1)^{m-1}\frac{B_{2m}}{m^2{2m\choose m}}X^{2m}
+\sum_{k=0}^{\lfloor m/2\rfloor}
\frac{B_{2k}B_{2(m-k)}}{m(m-k)}{m\choose 2k}X^{2k}
-\frac{1+(-1)^m}{2}\frac{B_m^2}{m^2}
\right ).$$
Then $R_m(1)=0$ for $m\geq 1$ and 
$S_m(d)=R_m(d)$ for $m\geq 1$ and $d\geq 2$.
\end{lemma}

\begin{proof}
For $d\geq 1$, set 
$$g_m(d,z) 
=d\cot (dz)\left (\cot^{(m-1)}(z)\right )^2
=\frac{1}{z^{2m+1}}(dz\cot (dz))(z^{m}\cot^{(m-1)}(z))^{2},$$
a $\pi$-periodic meromorphic function of the complex variable $z$, 
with a simple pole at each $k\pi/d$, $k\in {\mathbb Z}\setminus d{\mathbb Z}$, 
of residue 
\begin{equation}\label{residue}
{\rm Res}_{z=k\pi /d}(g_m(d,z)
=\left (\cot^{(m-1)}(k\pi /d)\right )^2
\ \ \ \ \ (k\in {\mathbb Z}\setminus d{\mathbb Z})
\end{equation}
and poles of order $2m+1$ at each $k\pi$, $k\in {\mathbb Z}$. 
For $R>0$, let $C_R$ be the positively oriented rectangle 
of width $\pi$, 
height $2R$ 
and center of gravity at $\pi/2-\pi/(2d)$.\\ 
(i). We claim that 
\begin{equation}\label{valueintCR}
\lim_{R\rightarrow +\infty}\frac{1}{2\pi i}\int_{C_R}g_m(d,z){\rm d}z
=\begin{cases}
-d&\hbox{if $m=1$ and $d\geq 1$,}\\
0&\hbox{if $m\geq 2$ and $d\geq 1$.}
\end{cases}
\end{equation}
Indeed, since $g$ is $\pi$-periodic the integrals over the vertical sides of $C_R$ cancel. 
Moreover, since $\vert\cot(t+ib)+i{\rm sign}(b)\vert\leq 2/(e^{\vert b\vert}-1)$ 
we have $\lim_{b\rightarrow+\infty} \cot (a+ib) =-i$ 
and $\lim_{R\rightarrow+\infty} \cot(a-ib) =i$ 
uniformly in $a\in I$, where $I$ is a given closed interval. 
Hence for $m=1$ we have
$\lim_{R\rightarrow+\infty} g_1(d,t+iR) =di$ 
and $\lim_{R\rightarrow+\infty} g_1(d,t-iR) =-di$, 
uniformly in $t\in [-\pi/(2d),\pi -\pi/(2d)]$. 
The first case of claim \eqref{valueintCR} follows. 
Now assume that $m\geq 2$. 
Then $\cot^{(m-1)} =Q_m(\cot)$ 
where the polynomials $Q_k(X)$'s are defined inductively by 
$Q_1(X)=X$ and $Q_{k+1}(X)=-(X^2+1)Q_k'(X)$.
Consequently, we have $\lim_{R\rightarrow+\infty} g_m(d,t+iR) =Q_m(-i)=0$ 
and $\lim_{R\rightarrow+\infty} g_m(d,t-iR) =Q_m(i)=0$, 
uniformly in $t\in [-\pi/(2d),\pi -\pi/(2d)]$. 
Hence the integrals over the horizontal sides tend to $0$ as $R$ goes to infinity for $m\geq 2$ 
and the second case of claim \eqref{valueintCR} follows.\\
(ii). Now, we claim that 
\begin{equation}\label{Rm(d)}
R_m(d)
=\lim_{R\rightarrow +\infty}\frac{1}{2\pi i}\int_{C_R}g_m(d,z){\rm d}z
-{\rm Res}_{z=0}(g_m(d,z))
\ \ \ \ \ \hbox{($m\geq 1$ and $d\geq 1$).}
\end{equation}
Indeed, if $m=1$ then 
$g_1(d,z) =d\cot(dz)\cot^2z$, 
$${\rm Res}_{z=0}(g_1(d,z)=-(d^2+2)/3$$ 
and the claims follow by \eqref{valueintCR} and the definition of $R_1(X)$.
Now assume that $m\geq2$. 
Using the power series expansion 
\begin{equation}\label{seriescotbis}
z\cot (z)
=\sum_{k\geq 0}(-1)^k\frac{2^{2k}B_{2k}}{(2k)!}z^{2k}
\ \ \ \ \ (\vert z\vert<\pi),
\end{equation} 
we have 
\begin{equation}\label{seriesderivativescot}
z^m\cot^{(m-1)}(z)
=(-1)^{m-1}(m-1)!
+\sum_{k\geq m/2}(-1)^k\frac{2^{2k}B_{2k}}{2k\cdot (2k-m)!}z^{2k}
\ \ \ \ \ (\vert z\vert<\pi).
\end{equation} 
Using \eqref{seriescotbis} and \eqref{seriesderivativescot}
to compute the coefficient of the term $z^{2m}$ in the power series expansion of the function
$(dz\cot (dz))(z^{m}\cot^{(m-1)}(z))^{2}$, 
we readily obtain that 
$${\rm Res}_{z=0}(g_m(d,z))
=-R_m(d)
\text{ for $m\geq 2$ and $d\geq 1$.}$$ 
Hence claim \eqref{Rm(d)} follows from \eqref{valueintCR}.\\
(iii). Finally, we use the residue theorem. 
For $d=1$, noticing that $z=0$ is the only pole of $g_m(1,z)$ inside $C_R$ and using \eqref{Rm(d)} 
we obtain $R_m(1)=0$ for $m\geq 1$.
For $d\geq 2$, noticing that $z=0$ and $z=k\pi/d$, $1\leq k\leq d-1$, 
are the only poles of $g_m(d,z)$ inside $C_R$ and using \eqref{Rm(d)} and \eqref{residue} 
we obtain $R_m(d)=S_m(d)$ for $d\geq 2$ and $m\geq 1$, as desired.
\end{proof}

In \cite[Corollary 6.8]{EL} the authors gave a reformulation of \cite[Corolloray 2.2]{BY}. 
It would be interesting to see if their method can also be applied to a reformulation of our Lemma \ref{TmSm}.

\section{A general explicit formula}
For $f>2$, $n\geq 2$ and $m_1,\cdots,m_n\geq 1$ such that $m_1+\cdots+m_{n}=2s$ is even, 
we set $\overrightarrow{m} =(m_1,\cdots,m_n)$ and
\begin{multline*}
M(\overrightarrow{m},f)
=\left (\frac{2}{\phi(f)}\right )^{n-1}\\
\times\sum_{\chi_1(-1)=(-1)^{m_1}}\cdots\sum_{\chi_{n-1}(-1)=(-1)^{m_{n-1}}}
L(m_1,\chi_1)\cdots L(m_{n-1},\chi_{n-1})\overline{L(m_n,\chi_1\cdots\chi_{n-1})}.
\end{multline*}
Hence, $(\chi_1\cdots\chi_{n-1})(-1) =(-1)^{m_n}$.
Using the formula 
$$L(m,\chi)
=\frac{(-1)^{m-1}\pi^m}{2\cdot f^m\cdot (m-1)!}
\sum_{k=1}^{f-1}\chi(k)\cot^{(m-1)}(\pi k/f)
\ \ \ \ \ \hbox{(whenever $\chi(-1)=(-1)^m$).}$$ 
proved in \cite[(1)]{LouCollMath90}, 
we have
\begin{multline*}
M(\overrightarrow{m},f)
=\frac{(-1)^{n}\pi^{2s}}
{2^nf^{2s}\prod_{l=1}^n (m_l-1)!}\\
\times
\sum_{k_n=1}^{f-1}\sum_{k_1=1}^{f-1}\cdots\sum_{k_{n-1}=1}^{f-1}
\left (\prod_{l=1}^{n}\cot^{(m_l-1)}(\pi k_l/f)\right )
\left (\prod_{l=1}^{n-1}\frac{2}{\phi(f)}\sum_{\chi_l(-1)=(-1)^{m_l}}\chi_l(k_l)\overline{\chi_l(k_n)}\right ).
\end{multline*}
We have the following orthogonality relations for the $\phi(f)/2$ characters $\chi_l$'s modulo $f>2$ 
for which $\chi_l(-1)=(-1)^{m_l}$:
$$\frac{2}{\phi(f)}\sum_{\chi_l(-1)=(-1)^{m_l}}\chi_l(k_l)\overline{\chi_l(k_n)}
=\begin{cases}
1&\hbox{if $k_l\equiv k_n\pmod f$ and $\gcd(k_n,f)=1$},\\
(-1)^{m_l}&\hbox{if $k_l\equiv -k_n\pmod f$ and $\gcd(k_n,f)=1$,}\\
0&\hbox{otherwise.}
\end{cases}$$
Noticing that $\chi_l(-k_n)\cot^{(m_l-1)}(-\pi k_n/f) =\chi_l(k_n)\cot^{(m_l-1)}(\pi k_n/f)$ 
whenever $\chi_l(-1)=(-1)^{m_l}$, 
we obtain 
$$M(\overrightarrow{m},f)
=\frac{(-1)^{n}\pi^{2s}}
{2f^{2s}\prod_{l=1}^n (m_l-1)!}
\sum_{\substack{k_n=1\\\gcd(k_n,f)=1}}^{f-1}
\prod_{l=1}^n
\cot^{(m_l-1)}(\pi k_n/f).$$
The following generalization of Lemma \ref{M(m,f)} follows:

\begin{lemma}\label{M(m,f)generalization}
For $f>2$, $n\geq 2$ and $m_1,\cdots,m_n\geq 1$ such that $m_1+\cdots+m_{n}=2s$ is even 
we have
$$M(\overrightarrow{m},f)
=\frac{(-1)^{n}\pi^{2s}}
{2\cdot (m_1-1)!\cdots (m_n-1)!\cdot f^{2s}}
\sum_{\substack{d\mid f\\ d>1}}\mu(f/d)S_{\overrightarrow{m}}(d),$$
where 
$$S_{\overrightarrow{m}}(d)
=\sum_{k=1}^{d-1}\prod_{l=1}^n
\cot^{(m_l-1)}\left (\frac{k\pi}{d}\right )
\ \ \ \ \ (d>1).$$
\end{lemma}

Now, for $d\geq 1$ a rational integer, we set 
$$g_{\overrightarrow{m}}(d,z) 
=d\cot (dz)\prod_{l=1}^n\cot^{(m_l-1)}(z)
=\frac{(dz\cot (dz))\prod_{l=1}^nz^{m_l}\cot^{(m_l-1)}(z)}{z^{2s+1}},$$
a $\pi$-periodic meromorphic function of the complex variable $z$, 
with a simple pole at each $k\pi/d$, $k\in {\mathbb Z}\setminus d{\mathbb Z}$, 
of residue 
\begin{equation}\label{residuegeneralization}
{\rm Res}_{z=k\pi /d}(g_{\overrightarrow{m}}(d,z)
=\prod_{l=1}^n\cot^{(m_l-1)}(k\pi /d)
\ \ \ \ \ (k\in {\mathbb Z}\setminus d{\mathbb Z})
\end{equation}
and poles of order $2s+1$ at each $k\pi$, $k\in {\mathbb Z}$. 
Let 
$T_{\overrightarrow{m}}(X)$ denote 
the coefficient of $z^{m_1+\cdots+m_n}$
in the formal power series expansion
$$\left (\sum_{k\geq 0}(-1)^k\frac{2^{2k}B_{2k}}{(2k)!}X^{2k}z^{2k}\right)
\prod_{l=1}^{n}
\left ((-1)^{m_l-1}(m_l-1)!
+\sum_{k\geq m_l/2}(-1)^k\frac{2^{2k}B_{2k}}{2k\cdot (2k-m_l)!}z^{2k}\right ).$$
Hence, $T_{\overrightarrow{m}}(X)$ is a polynomial of degree $2s=m_1+\cdots+m_n$, 
with rational coefficients and leading coefficient 
$$(-1)^{s-n}
\frac{2^{2s}B_{2s}}{(2s)!}
\prod_{l=1}^n(m_l-1)!$$
and
\begin{equation}\label{ResidueTm}
{\rm Res}_{z=0}(g_{\overrightarrow{m}}(d,z))
=T_{\overrightarrow{m}}(d)\hbox{ for $d\geq 1$}.
\end{equation}
With the same proof, we have the following generalization of \eqref{valueintCR}:
\begin{equation}\label{valueintCRgeneralization}
\lim_{R\rightarrow +\infty}\frac{1}{2\pi i}\int_{C_R}g_{\overrightarrow{m}}(d,z){\rm d}z 
=\begin{cases}
(-1)^{n/2}d&\hbox{(if $m_1=\cdots=m_n=1$),}\\ 
0& \hbox{otherwise.}
\end{cases}
\end{equation}
($n$ must be even if $m_1=\cdots=m_n=1$).\\
We have the following generalization of Lemma \ref{TmSm}:

\begin{lemma}\label{TmSmgeneralization}
Set 
$$R_{\overrightarrow{m}}(X)
=-T_{\overrightarrow{m}}(X)
+\begin{cases}
(-1)^{n/2}X&\hbox{(if $m_1=\cdots=m_n=1$),}\\ 
0& \hbox{otherwise.}
\end{cases}$$
Then 
$R_{\overrightarrow{m}}(1)=0$ 
and $S_{\overrightarrow{m}}(d)=R_{\overrightarrow{m}}(d)$ for $d\geq 2$.
\end{lemma}

\begin{proof}
By \eqref{ResidueTm} and \eqref{valueintCRgeneralization}, for $d\geq 1$ we have 
$$R_{\overrightarrow{m}}(d)
:=\frac{1}{2\pi i}\int_{C_R}g_{\overrightarrow{m}}(d,z){\rm d}z
-{\rm Res}_{z=0}(g_{\overrightarrow{m}}(d,z)).$$
We use the residue theorem.
For $d=1$, noticing that $z=0$ is the only pole of $g_{\overrightarrow{m}}(1,z) $ inside $C_R$, 
here again we get
\begin{equation}\label{R(m1mn(1)}
R_{\overrightarrow{m}}(1)=0.
\end{equation}
For $d\geq 2$, noticing that $z=0$ and $z=k\pi/d$, $1\leq k\leq d-1$, 
are the only poles of $g_{\overrightarrow{m}}(d,z)$ inside $C_R$ 
and using \eqref{residuegeneralization}, we obtain 
$R_{\overrightarrow{m}}d)
=S_{\overrightarrow{m}}d)$. 
\end{proof}

By Lemma \ref{TmSmgeneralization}, Lemma \ref{Rk} and \eqref{zeta(2m)}, 
and noticing that $T_{\overrightarrow{m}}(X)$ is even 
we end up with the following generalization of Theorem \ref{mainresult}:

\begin{theorem}\label{mainresultgeneralization}
Write
$$T_{\overrightarrow{m}}(X)
=(-1)^{s-n}
\frac{2^{2s}B_{2s}}{(2s)!}
\left (\prod_{l=1}^n(m_l-1)!\right )
\left (X^{2s}+\sum_{k=0}^{s-1}r_{\overrightarrow{m},2k}X^{2k}\right ),$$
Then, for $f>2$ we have
\begin{multline*}
M(\overrightarrow{m},f)
=\zeta(2s)
\times\left (\phi_{2s}(f)
+\sum_{k=1}^{s-1}r_{\overrightarrow{m},2k}
\frac{\phi_{2k}(f)}{f^{2s-2k}}
\right )\\
+\begin{cases}
(-1)^{n/2}\frac{\pi^n\phi_1(f)}{2f^{n-1}}&\hbox{(if $m_1=\cdots=m_n=1$),}\\ 
0& \hbox{otherwise.}
\end{cases}
\end{multline*}
In particular, as $f$ goes to infinity $M(\overrightarrow{m},f)$ is asymptotic to 
$$\zeta(2s)\phi_{2s}(f),
\hbox{ where }
2s=m_1+\cdots+m_n.$$ 
\end{theorem}

\section{Several explicit examples}
\subsection{First example.} 
By Lemma \ref{TmSmgeneralization} we have 
$S_{\overrightarrow{(1,1,2)}}(d)
=R_{\overrightarrow{(1,1,2)}}(d)
=-T_{\overrightarrow{(1,1,2)}}(d)$. 
By \eqref{ResidueTm}, we obtain
$$S_{\overrightarrow{(1,1,2)}}(d)
=-T_{\overrightarrow{(1,1,2)}}(d)
=-{\rm Res}_{z=0}(d\cot(dz)\cot^2(z)\cot'(z))
=-\frac{1}{45}d^4+\frac{1}{9}d^2-\frac{4}{45}.$$
Therefore 
$$M(\overrightarrow{(1,1,2)},f)
=-\frac{\pi^4}{2f^4}\sum_{d\mid f}\mu (f/d)S_{\overrightarrow{(1,1,2)}}(d)
=\frac{\pi^4}{90}\left (\phi_4(f)-5\frac{\phi_2(f)}{f^2}\right ),$$
by Lemma \ref{M(m,f)generalization} and Lemma \ref{Rk}.
We have recovered \cite[Theorem 1]{Alkan}.

\subsection{Second example.} 
\begin{theorem}
Let $n\geq 2$ be even. 
For $f>2$ set 
$$M_n(f)
=\left (\frac{2}{\phi(f)}\right )^{n-1}
\sum_{\chi_1(-1)=-1}\cdots\sum_{\chi_{n-1}(-1)=-1}
L(1,\chi_1)\cdots L(1,\chi_{n-1})\overline{L(1,\chi_1\cdots\chi_{n-1})}.$$
Then 
$$M_n(f)
=\zeta(n)\times\left (
\phi_n(f)
+\frac{n!}{B_n}\sum_{k_0=1}^{n/2-1}
C_{n,n/2-k_0}
\frac{B_{2k_0}}{(2k_0)!}
\frac{\phi_{2k_0}(f)}{f^{n-2k_0}}
\right )
+(-1)^{n/2}\frac{\pi^n\phi_1(f)}{2f^{n-1}},$$
where 
$$C_{n,N}
=\sum_{\substack{e_1,\cdots,e_N\geq 0\\ e_1+2e_2\cdots+Ne_N=N}}
\frac{n!}{(n-(e_1+\cdots+e_N))!}
\prod_{l=1}^N\frac{1}{e_l!}\left (\frac{B_{2l}}{(2l)!}\right )^{e_l}
\ \ \ \ \ (n\geq N\geq 1).$$
\end{theorem}

\begin{proof}
Here 
$$T_{\overrightarrow{m}}(d)
=(-1)^{n/2}\frac{2^{n}B_n}{n!}\sum_{\substack{k_0,k_1,\cdots,k_n\geq 0\\ k_0+k_1+\cdots+k_n=n/2}}
\frac{n!}{B_n}\left (\prod_{l=1}^n\frac{B_{2k_l}}{(2k_l)!}\right )\frac{B_{2k_0}}{(2k_0)!}d^{2k_0}.$$
Now, for $n\geq N\geq 1$ and $F:{\mathbb Z}_{\geq 0}\rightarrow {\mathbb C}^*$ such that $F(0)=1$, 
e.g. for $F(k) =B_{2k}/(2k)!$, 
we have 
$$\sum_{\substack{k_1,\cdots,k_n\geq 0\\ k_1+\cdots+k_n=N}}\prod_{l=1}^nF(k_l)
=\sum_{\substack{e_1,\cdots,e_N\geq 0\\ e_1+2e_2\cdots+Ne_N=N}}
\frac{n!}{(n-(e_1+\cdots+e_N))!}\prod_{l=1}^N\frac{F(l)^{e_l}}{e_l!},$$
by letting $e_l$ denote the number of $k_i$'s which are equal to $l$ 
and by noticing that we have ${n\choose e_0}$ possible choices of the $k_i$'s which are equal to $0$, 
then ${n-e_0\choose e_1}$ possible choices of the $k_i$'s which are equal to $1$,..., 
and finally by using $e_0+e_1+\cdots e_N =n$ and
$${n\choose e_0}{n-e_0\choose e_1}\cdots{n-(e_0+e_1+\cdots+e_{N-1})\choose e_N}
=\frac{n!}{e_0!\times e_1!\times\cdots\times e_N!}.$$
\end{proof}

For example, 
$$M_2(f)
=\frac{\pi^2}{6}\left (\phi_2(f)-3\frac{\phi_1(f)}{f}\right ),$$ 
in accordance with \eqref{M(1,f)},
$$M_4(f)
=\frac{\pi^4}{90}
\left (\phi_4(f)-20\frac{\phi_2(f)}{f^2}+45\frac{\phi_1(f)}{f^3}\right )$$
and
$$M_6(f)
=\frac{\pi^4}{945}
\left (
\phi_6(f)-21\frac{\phi_4(f)}{f^2}+\frac{483}{2}\frac{\phi_2(f)}{f^4}-\frac{945}{2}\frac{\phi_1(f)}{f^5}
\right )$$

For $f=3$, there is only one odd character $\chi$ modulo $3$, 
and $L(1,\chi) =\frac{\pi}{3\sqrt 3}$. 
Hence, $M_n(3) =\frac{\pi^n}{3^{3n/2}}$. 
Our formulae do give 
$M_2(3) =\frac{\pi^2}{3^3}$, 
$M_4(3) =\frac{\pi^4}{3^6}$ 
and $M_6(3) =\frac{\pi^6}{3^9}$.

\subsection{Third example.} 
We have the following generalization of Theorem \ref{mainresult}:

\begin{theorem}\label{M(m,n;f)}
For $m\geq 1$, $n\geq 1$ of the same parity and $(m,n)\neq (1,1)$, we have
\begin{align*}
M(\overrightarrow{(m,n)},f)
=&\frac{2}{\phi(f)}\sum_{\chi(-1)=(-1)^m=(-1)^n}L(m,\chi)\overline{L(n,\chi)}\\
=&\zeta(m+n)
\times\Biggl (\phi_{m+n}(f)\\
&+(-1)^{m-1}{m+n\choose m}\sum_{k=1}^{\lfloor m/2\rfloor}
\frac{n}{m+n-2k}{m\choose 2k}\frac{B_{2k}B_{m+n-2k}}{B_{m+n}}\frac{\phi_{2k}(f)}{f^{m+n-2k}}\\
&+(-1)^{n-1}{m+n\choose n}\sum_{k=1}^{\lfloor n/2\rfloor}
\frac{m}{m+n-2k}{n\choose 2k}\frac{B_{2k}B_{m+n-2k}}{B_{m+n}}\frac{\phi_{2k}(f)}{f^{m+n-2k}}
\Biggr ).
\end{align*}
In particular, $M(\overrightarrow{(m,n)},f)$ is asymptotic to $\zeta(m+n)\phi_{m+n}(f)$ 
as $f$ goes to infinity. 
\end{theorem}

\begin{proof}
Use Lemma \ref{TmSmgeneralization}, 
notice that 
\begin{align*}
T_{\overrightarrow{(m,n)}}(X)
=2^{m+n}
\times&\Biggl ((-1)^{\frac{m+n}{2}}\frac{B_{m+n}}{mn{m+n\choose n}}X^{m+n}\\
&-(-1)^{\frac{m-n}{2}}\sum_{k=1}^{\lfloor m/2\rfloor}
\frac{B_{2k}B_{m+n-2k}}{m(m+n-2k)}{m\choose 2k}X^{2k}\\
&-(-1)^{\frac{n-m}{2}}\sum_{k=1}^{\lfloor n/2\rfloor}
\frac{B_{2k}B_{m+n-2k}}{n(m+n-2k)}{n\choose 2k}X^{2k}\\
&+(-1)^{\frac{m+n}{2}}\frac{1+(-1)^m}{2}\frac{1+(-1)^m}{2}\frac{B_mB_n}{mn}
\Biggr ).
\end{align*}
and finally use Lemma \ref{M(m,f)generalization}.
\end{proof}

Our explicit formula for $M(\overrightarrow{(m,n)},f)$ is much simpler 
than the ones given in \cite[Theorem 1.1]{LZ} and \cite[Theorem 1.2]{OO}. 
For example, the first term in the formula given \cite[Theorem 1.1]{LZ} 
is 
$$(-1)^{(m-n)/2}\frac{(2\pi)^{m+n}}{2\cdot m!\cdot n!}\phi_{m+n}(f)
\sum_{a=0}^m\sum_{b=0}^nB_{m-a}B_{n-b}
\frac{{m\choose a}{n\choose b}}{a+b+1}.$$ 
By \eqref{zeta(2m)}, 
this first term being equal to the one $\zeta(m+n)\phi_{m+n}(f)$ of Theorem \ref{M(m,n;f)} 
is tantamount to saying that 
$B_{m+n}=(-1)^n{m+n\choose n}\sum_{a=0}^m\sum_{b=0}^nB_{m-a}B_{n-b}$. 
This is probably correct and at least we checked it for several values of $m$ and $n$.

Notice that for $m\geq 2$ we have $M(\overrightarrow{(m,m)},f) =M(m,f)$ 
and we recover the formula given in Theorem \ref{mainresult}.

\subsection{Fourth example.} 
For $m,n\geq 1$, Theorem \ref{mainresultgeneralization} gives a formula for
$$M(\overrightarrow{(m,n,m+n)},f)
=\frac{4}{\phi(f)^2}
\sum_{\chi_1(-1)=(-1)^m}\sum_{\chi_2(-1)=(-1)^n}L(m,\chi_1)L(n,\chi_2)\overline{L(m+n,\chi_1\chi_2)}$$
much simpler than the one given in \cite[Theorem 1.1]{OO}.

\bibliography{central}

\begin{thebibliography}{??????}
\bibitem[Alk]{Alkan}
E. Alkan. 
\newblock Averages of values of $L$-series. 
\newblock {\em Proc. Amer. Math. Soc.} {\bf 141} (2013), 1161--1175.

\bibitem[BY]{BY}
Bruce C. Berndt and Boon Pin Yeap. 
\newblock Explicit evaluations and reciprocity theorems for finite trigonometric sums. 
\newblock {\em Adv. in Appl. Math.} {\bf 29} (2002), 358--385.

\bibitem[EL]{EL}
W. Ejsmont and F. Lehner.
\newblock The trace method for cotangent sums.
\newblock {\em J. Combin. Theory Ser. A} {\bf 177} (2021), Paper No. 105324, 32 pp. 

\bibitem[Lin]{Lin}
Xin Lin.
\newblock On the mean square value of the Dirichlet L-function at positive integers. 
\newblock {\em Acta Arith.} {\bf 189} (2019), 367--379. 

\bibitem[Lou93]{Lou93}
S. Louboutin.
\newblock Quelques formules exactes pour des moyennes de fonctions $L$ de Dirichlet.
\newblock {\em Canad. Math. Bull.} {\bf 36} (1993), 190--196. 

\bibitem[Lou94]{Lou94}
S. Louboutin.
\newblock Corrections \`a : Quelques formules exactes pour des moyennes de fonctions $L$ de Dirichlet.
\newblock {\em Canad. Math. Bull.} {\bf 37} (1994), p. 89. 

\bibitem[Lou01]{LouCollMath90}
S. Louboutin.
\newblock The mean value of $\vert L(k,\chi )\vert^2$ at positive rational integers $k\geq 1$.
\newblock {\em Colloq. Math.} {\bf 90} (2001), 69--76.

\bibitem[LZ]{LZ}
H. Liu and W. Zhang. 
\newblock On the mean value of $L(m,\chi)L(n,\bar\chi)$ at positive integers $m,n\geq 1$. 
\newblock {\em Acta Arith.} {\bf 122} (2006), 51--56.

\bibitem[OO]{OO}
T. Okamoto and T. Onozuka.
\newblock On various mean values of Dirichlet $L$-functions.
\newblock {\em Acta Arith.} {\bf 167} (2015), 101--115.

\bibitem[QiMG]{Qi}
Qi Ming Gao.
\newblock A kind of mean square formula for $L$-functions. (in Chinese)
\newblock {\em J. Tsinghua Univ.} {\bf 31} (1991), 34-41. See [MR1168609(93g:11090).

\bibitem[Wal]{Wal}
H. Walum.
\newblock An exact formula for an average of $L$-series.
\newblock {\em Illinois J. of Math.} {\bf 26} (1982), 1-3. 
\end{thebibliography}

\end{document}